\documentclass{l4dc2024}

\usepackage{amsmath}
\usepackage{amssymb}
\usepackage[utf8]{inputenc}
\usepackage{dsfont}
\usepackage{color}

\usepackage{graphicx}
\usepackage{epstopdf}

\newtheorem{thm}{Theorem}

\newtheorem{lem}{Lemma}
\newtheorem{rem}{Remark}

\newtheorem{ex}{Example}

\newcommand{\R}{\mathbb{R}}

\newcommand{\N}{\mathbb{N}}

\newcommand{\eps}{\varepsilon}

\title[Learning Polynomial Representations of Physical Objects]{Learning Polynomial Representations of Physical Objects with Application to Certifying Correct Packing Configurations}
\usepackage{times}

\author{%
	\Name{Morgan Jones} \Email{morgan.jones@sheffield.ac.uk}\\
	\addr The Department of Automatic Control and Systems Engineering,
	The University of Sheffield
}

\begin{document}
	\maketitle
	
	\begin{abstract}%
		This paper introduces a novel approach for learning polynomial representations of physical objects. Given a point cloud data set associated with a physical object, we solve a one-class classification problem to bound the data points by a polynomial sublevel set while harnessing Sum-of-Squares (SOS) programming to enforce prior shape knowledge constraints. By representing objects as polynomial sublevel sets we further show it is possible to construct a secondary SOS program to certify whether objects are packed correctly, that is object boundaries do not overlap and are inside some container set. While not employing reinforcement learning (RL) in this work, our proposed secondary SOS program does provide a potential surrogate reward function for RL algorithms, autonomously rewarding agents that propose object rotations and translations that correctly pack objects within a given container set.



		
	\end{abstract}
	
		 \vspace{-0.25cm}
	\section{Introduction}
		 \vspace{-0.1cm}
	Emerging Machine Learning (ML) techniques have the potential to cause substantial impacts on society. However, deployment of these methods in critical decision-making processes may be hindered by concerns about their reliability, safety, and correctness~\cite{kwiatkowska2023trust,tambon2022certify}. In this paper we propose a general framework of autonomously certifying the correctness of ML outputs involving objects belonging to physical space. Our approach begins by representing physical objects as polynomial sublevel sets. Subsequently, this allows us to formulate the certification that a given object configuration satisfies certain properties as a Sum-of-Squares (SOS) optimization problem. More specifically, we show how this framework allows us to certify feasible solutions to the packing problem, the problem of arranging objects in a non-overlapping configuration within a container set.

	Methods that seek to solve the packing problem are of significant importance to E-commerce, or electronic commerce. E-commerce has radically transformed the way we shop and do business. Many companies now use e-commerce by incorporating technologies like online payment systems and virtual marketplaces into their platforms. One such important problem arising from E-commerce is the problem of reducing package waste. Many products purchased online are shipped individually in boxes that are often much larger than the products themselves, filled with packing materials such as bubble wrap, Styrofoam, or air pillows, which are not biodegradable and are not easily recyclable~\cite{chueamuangphan2020packaging}. It is estimated that online shopping generates 4.8 times more packaging
	waste than offline shopping~\cite{KIM2022110398} and more than 900 million packages are ordered online each year in the UK~\cite{allen2017analysis}. Package waste causes substantial environmental damage by filling up landfills and increasing carbon emissions through their production, disposal and transportation~\cite{wang2020you}. Amazon has investigated the use of ML to predict the optimal container material for specific objects~\cite{meiyappan2021position}. Once a container and packing material is selected we must then solve the packing problem, correctly arranging a set of objects within the container.
	
	
	 For general objects and container sets, solving the packing problem is a formidable challenge. In niche special cases, such as fitting eleven or less 2D circles into a larger circular container set, analytical optimal configurations are known~\cite{huang2011global}. However, even in other simple cases where the problem is limited to two dimensions, all objects are identical squares, and the container is a convex shape, it is widely recognized that the problem is NP-Complete~\cite{fowler1981optimal}. Most existing techniques focus on solving the packing problem are restricted to simple shaped objects (rectangles, ellipsoids, etc)~\cite{fang2023deep,tu2023deep,yang2021packerbot,que2023solving,xiong2023towards,hopper1999genetic,yang2023heuristics}. There are far fewer existing methods that seek to pack general shapes, such as those found in~\cite{pan2023sdf,zhao2023learning,zhao2021optimizing,liu2023optimizing,huang2022planning}. All of these methods are specialized to two or three dimensions and require some sort of discretization to ensure objects do not overlap and the packing configuration is correct. The method proposed in this paper works for arbitrary dimensions, requires no discretization and compliments existing packing methods by providing a way to autonomously and mathematically certify that a packing configuration is correct. By expressing objects and container sets as polynomial sublevel sets, we show that there exists an associated SOS program capable of certifying the correctness of a given packing configuration.
	 
To implement our proposed SOS packing certification program we must first find polynomial sublevel representations of physical objects and containers. In this work we find these polynomial representations by reconstructing a surface of an object/container from point cloud data. Surface reconstruction plays a pivotal role in various fields, including virtual reality~\cite{sorokin20233d}, 3D printing~\cite{sheng2018lightweight}, robotics~\cite{nigro2023assembly}, and medical imaging~\cite{bernard2017shape}. There is an extensive literature on non-polynomial surface reconstruction methods~\cite{mishra2023convolutional,wu2023category,ma2022surface,huang2022surface,bruel2020topology}. There is much less existing work on polynomial surface reconstruction methods. Many polynomial set estimation methods revolve around approximating semialgebraic sets and their associated Minkowski sums and Pontryagin differences, see~\cite{dabbene2017simple,cotorruelo2022sum,guthrie2022inner,guthrie2022closed,marschner2021sum,habibi2023polynomial,jones2019using}. To the best of the author's knowledge, the only existing methods for polynomial surface reconstruction based on point cloud data are~\cite{magnani2005tractable,ahmadi2016geometry,jones2023sublevel}. Of these methods \cite{magnani2005tractable,ahmadi2016geometry} are based on heuristic objectives while \cite{jones2023sublevel} has convergence guarantees. In this work, we modify the method in \cite{jones2023sublevel} to allow for prior knowledge shape SOS constraints, such as symmetry's, star connectedness and convexity. The approach of enforcing shape constraints shares some similarity to~\cite{pitarch2019systematic} that solved regression problems with SOS model constraints and  \cite{ahmadi2020learning,machado2023sparse} that solved system identification problems while enforcing model ``side information" or prior model knowledge SOS constraints.

	 \vspace{-0.35cm}
	 \paragraph{Notation:} For $x \in \R^n$ we denote the Euclidean norm as $||x||_2:=\sqrt{\sum_{i=1}^n x_i^2}$. We denote the $n$-dimensional ball of radius $r>0$ and centred at $y \in \R^n$ by $B_r(0)=\{x \in \R^n: ||x-y||_2^2<r^2 \}$. We denote the space of polynomials $p: \R^n \to \R$ by $\R[x]$ and polynomials with degree at most $d \in \N$ by $\R_d[x]$. Similarly, we denote  $\R^{s \times n}[x]$ to be the set of matrix polynomials $p: \R^n \to \R^{s \times n}$. We say $p \in \R_{2d}[x]$ is Sum-of-Squares (SOS) if there exists $p_i \in \R_{d}[x]$ such that $p(x) = \sum_{i=1}^{k} (p_i(x))^2$. We denote $\sum_{SOS}^{2d}$ to be the set of SOS polynomials of at most degree $d \in \N$ and the set of all SOS polynomials as $\sum_{SOS}$. We use the symbol $ \exists $ to mean ``There exists".

 \section{Learning Polynomial Representations of Physical Objects}


In this work we assume that we have access to a point cloud data set, a collection of data points $\{x_i\}_{i=1}^N \subset \R^n$, providing coordinates and position of the surface of the object/container. Such point clouds can be generated through various methods, such as 3D scanning technologies like LiDAR (Light Detection and Ranging) or photogrammetry. Apart from point cloud data, we also assume that we have prior knowledge about the shape of the objects/container sets such as symmetries, convexity and connectedness. 

The work of~\cite{jones2023sublevel} proposed a method that computes outer polynomial sublevel sets of discrete points by approximating the function $V(x)= 1- \mathds{1}_{\{x_i\}_{i=1}^N}(x)$. We next adapt the SOS program from~\cite{jones2023sublevel} by adding prior knowledge shape constraints, 
\begin{align} \label{opt: SOS discrete points}
	J_d^* \in & \arg \sup_{J_d \in \R_d[x]}  \int_\Lambda J_d(x) dx \qquad \text{ such that }\\ \nonumber
	& J_d(x_i) <0  \text{ for } 1 \le i \le N, \\ \nonumber 
	&  1-J_d(x) - s_0(x)(R^2 -||x||_2^2) \in \sum_{SOS}^d, \quad s_0(x) \in \sum_{SOS}^d,\\ \nonumber
	& J \in \mathcal{S}_i \text{ for all } 1 \le i \le M,
\end{align}
where $\Lambda=[a_1, b_1] \times \dots \times [a_n,b_n]$ is chosen as a rectangular region. This choice enables us to analytically pre-calculate the integral of the monomial basis functions of $J_d$ over $\Lambda$. As a result, the objective function in Opt.~\eqref{opt: SOS discrete points} reduces to a simple weighted linear sum of coefficients of $J_d$.
%
%
Moreover, we select $\Lambda$ such that $\{x_i\}_{i=1}^N \subset  \Lambda$ and select $R>0$ sufficiently large so that $\Lambda \subseteq B_R(0)$ . Finally, $\mathcal{S}_i$ are sets of functions that satisfy certain SOS constraints relating to prior shape knowledge. In particular, we consider the following prior shape knowledge constraints:
\begin{align} \label{prior knowledge: symetry}
	&	\text{Symmetry: } \mathcal{S}=\{J_d \in \R[x]: J(x)=J(Ax) \}, \text{ defined for some } A\in \R^{n \times n} . \\ \label{prior knowledge: star}
	&	\text{Star Connectedness: } \mathcal{S}=\{J_d \in \R[x]: x \nabla J_d(x) >0  \}. \\ \label{prior knowledge: convexity}
	&	 \text{Convexity: } \mathcal{S}=\{J_d \in \R[x]: \nabla^2 J_d(x) >0  \}.
\end{align}
To enforce the prior shape knowledge constraints given in Eqs~\eqref{prior knowledge: symetry} and \eqref{prior knowledge: star} we tighten the polynomial inequalities to SOS constraints while using Putinar's Positivstellesatz (Thm.~\ref{thm: Psatz}) to enforce these constraints locally over the ball of radius $R>0$. To enforce Eq.~\eqref{prior knowledge: convexity} we tighten the inequality to an SOS matrix inequality, that is we follow~\cite{ahmadi2013complete} and constrain $\nabla^2 J_d(x)=M(x)^\top M(x)$ for some $M(x) \in \R^{s \times n}[x]$ and $s\in \N$.
\begin{lem}
	Suppose $J_d^*$ solves Opt.~\eqref{opt: SOS discrete points}. Denote $X_d:=\{x \in \Lambda : J_d(x) < 0 \}$. Then, 
	\begin{enumerate}
		\item $\{x_i\}_{i=1}^N \subset X_d$.
		\item If $\mathcal{S}=\emptyset$ then $X_d \to \{x_i\}_{i=1}^N$ in the volume metric.
		\item If Eq.~\eqref{prior knowledge: symetry} is enforced then if $x \in X_d$ it follows $Ax \in X_d$.
		\item If Eq.~\eqref{prior knowledge: star} is enforced then if $x \in X_d$ it follows $\lambda x \in X_d$ for all $\lambda \in (0,1]$.
		\item If Eq.~\eqref{prior knowledge: convexity} is enforced then if $x,y \in X_d$ it follows $\alpha x + (1-\alpha)y \in X_d$ for all $\alpha \in [0,1]$.
	\end{enumerate}
\end{lem} \label{lem: sublevel set approx propoerties}
\begin{proof}
	Since $J_d(x_i)<0$ for $1 \le i \le N$ is enforced in Opt.~\eqref{opt: SOS discrete points} it trivially follows $\{x_i\}_{i=1}^N \subset X_d$. When $\mathcal{S}=\emptyset$ it follows by~\cite{jones2023sublevel} that $X_d \to \{x_i\}_{i=1}^N$. If Eq.~\eqref{prior knowledge: symetry} is enforced then if $x \in X_d$ it follows $J_d(Ax)=J_d(x)<0$ implying $Ax \in X_d$. If Eq.~\eqref{prior knowledge: star} is enforced then if $x \in X_d$ it follows $\lambda x \in X_d$ for all $\lambda \in (0,1]$ follows from Lemma 9 in~\cite{wang2005polynomial}. Finally it follows by Prop 1.2.6 from~\cite{bertsekas2003convex} that if Eq.~\eqref{prior knowledge: convexity} holds then $J_d^*$ is convex and hence its sublevel sets are convex. 
\end{proof}
\begin{rem}
	In \cite{jones2023sublevel} it was shown that when $\mathcal{S}_i=\emptyset$ the integral objective of Opt.~\eqref{opt: SOS discrete points} ensures that this bound is tight. That is, without prior knowledge shape constraints $\{x \in \Lambda: J_d^*(x) \le 0 \} \to \{x_i\}_{i=1}^N$ in the volume metric. By adding prior knowledge shape constraints we will loose this asymptotically optimal bound on the point cloud. For instance, the set of discrete points $\{x_i\}_{i=1}^N$ is never convex for $N \ge 2$, yet we may constrain the sublevel set of $J^*_d$ to be convex, hence in this scenario it is clear that $\{x \in \Lambda: J_d^*(x) \le 0 \} \not\to \{x_i\}_{i=1}^N$.
	However, in practice our goal is to approximate the physical object/container set and not the point cloud.  The addition of prior knowledge shape constraints is useful in achieving this goal. 
\end{rem}
	 
	We have seen how to reconstruct polynomial sublevel set representations, $\{x \in \Lambda: J_d^*(x) \le 0 \}$, of objects by solving Opt.~\eqref{opt: SOS discrete points}. The rectangular set, $\Lambda$, can be expressed as a semialgebraic set but it is hard to express it as a single polynomial sublevel set. However, since $\Lambda$ is selected sufficiently large to contain point cloud data, typically there exists $r>0$ such that $B_r(0) \subset \Lambda$ and
	\begin{equation} \label{eq: obj poly sublevel set rep}
		\{x \in \Lambda: J_d^*(x) \le 0 \}=\{x \in B_r(0) : J_d^*(x) \le 0\}
	\end{equation}
Therefore, to simply notation in the proceeding section we will assume objects take the form given in Eq.~\eqref{eq: obj poly sublevel set rep}, that is objects are represented by $\{x \in \R^n: J(x) \le 0, F(x) \le 0\}$, for some $J,F \in \R[x]$. 
	 \vspace{-0.3cm}
	 \section{A Set Formulation of The Packing Problem}
 Consider a container, $C \subset \R^n$, and object sets to be packed, $P_i \subset \R^n$, then  the packing problem can be formulated as the following feasibility problem,
	\vspace{-0.2cm} \begin{align} \label{opt: general packing}
		& \text{Find }_{T_i \in E(n)}\\ \nonumber 
		& \text{Subject to: } T_i(P_i) \subset C \text{ for } i \in \{1,\dots,N\} \text{ and } T_i(P_i) \cap T_j(P_j) = \emptyset \text{ for all } i \ne j.
	\end{align}\vspace{-0.6cm}\text{ }
	
	\noindent
	Where $E(n)$ denotes the euclidean group containing all rigid transformations that preserve distance and for $T \in E(n)$, and for each $P \subset \R^n$ we define the transformation $T(P):=\{T(p): p \in P \} \subset \R^n$. In matrix and vector notation, the transformation $x \to Tx$ for some $x \in \R^n$ and $T \in E(n)$ is equivalent to $x \to A(x +v)$ where $A \in \R^{n \times n}$ is some orthogonal matrix ($A^\top A=AA^\top=I$) and $v \in \R^n$ is a finite magnitude vector.
	
	 To solve the Optimization Problem~\eqref{opt: general packing} we must find transformations, $T_i \in E(n)$, of how we rotate, translate and reflect each object, $P_i$, in order for the object to be placed in the container, that is $T_i(P_i) \subset C$, and in a way that respects the volume boundaries of each objects, that is objects must not overlap when placed  $T_i(P_i) \cap T_j(P_j) = \emptyset$.
	
	We should remark that Optimization Problem~\eqref{opt: general packing} can easily be extended to the more challenging case where both the container, $C \subset \R^n$, and the objects, $P_i \subset \R^n$, are decision variables. That is we must choose the maximum number of objects that can be placed in the container with minimum volume. However, Optimization Problem~\eqref{opt: general packing} is extremely challenging to solve. Even when the rigid transformations, objects and container sets are all given and fixed it is still difficult to certify whether or not the given solution is feasible. In the following section we will show that if we model the object and container sets as polynomial sublevel sets then we can certify that a given packing configuration is correct using convex optimization. 
	
	\vspace{-0.3cm}
\subsection{Formulating The Packing Problem as a Polynomial Optimization Problem}
\vspace{-0.1cm}
Finding the optimal solution for the packing problem can require significant computational resources. 
The constraints in Opt.~\eqref{opt: general packing} involve sets, which are mathematical collections of
uncountable elements. This poses a challenge as it is difficult to
search over these uncountable elements and ensure object sets do not overlap and are within the container set. To make such
set optimization problems tractable, we need to find ways to
parametrize these sets. The approach we take to
overcome this problem is to consider sets given by polynomial sublevel sets. This has the following advantages:
\vspace{-0.1cm}
\begin{enumerate}
	\item Polynomials sublevel sets are parametrized by a finite vector of monomial coefficients that can be efficiently searched over and stored in computer memory. 
	\vspace{-0.2cm}
	\item We can represent any compact container and object set to arbitrary accuracy by a polynomial sublevel set. See Lemma~\ref{lem: any compact set can be approx by poly} in the Appendix. 
	\vspace{-0.2cm}
	\item Given point cloud data of objects and containers we can compute approximate polynomial representations by solving Opt.~\eqref{opt: SOS discrete points}.
\end{enumerate}
\vspace{-0.2cm}
By expressing the objects and container as polynomial sublevel sets in the following way $P_i := \{x \in \mathbb{R}^n : p_i(x) \le 0, F_i(x) \le 0 \}$ and $C := \{x \in \mathbb{R}^n : c(x) < 0, F_0(x) < 0 \}$, where the sublevel set of $F_i \in \mathbb{R}[x]$ is some computational domain such as $B_r(0) \subset \Lambda$ from Opt.~\ref{opt: SOS discrete points}, we can reframe Opt.~\eqref{opt: general packing} in the following way,
\vspace{-0.2cm} \begin{align} \label{opt: poly ineq}
	& \text{Find }_{T_i \in E(n)}\\ \nonumber 
	& \text{Subject to: }   c(x) < 0 \text{ for all } x \in \{ y \in \R^n :F_i(T_i^{-1}y) \le 0, p_i(T_i^{-1}y ) \le 0\} \text{ and } i \in \{1,\dots,N\},  \\ \nonumber 
	& \hspace{0.5cm} F_0(x) < 0 \text{ for all } x \in \{ y \in \R^n :F_i(T_i^{-1}y) \le 0, p_i(T_i^{-1}y ) \le 0\} \text{ and } i \in \{1,\dots,N\},  \\ \nonumber 
	& \hspace{0.5cm} p_j(T_j^{-1}x ) > 0 \text{ for all } x \in \{y \in \R^n: F_i(T_i^{-1}y) \le 0, p_i(T_i^{-1}y) \le 0,F_j(T_j^{-1}y) \le 0 \} \text{ and } i \ne j,
\end{align}
where we define the inverse of a Euclidean group transformation in the following way: If $T_i \in E(n)$ then there exists an orthogonal matrix $A_i \in \R^{n \times n}$ and translation vector $v_i \in \R^n$ such that $T_i(P_i)=\{A_i(x+v_i):x \in P_i\}=\{x\in\R^n: F_i(A_i^\top(x-v_i)) \le 0, p_i(A_i^\top(x-v_i)) \le 0   \} =: \{ x \in \R^n :F_i(T_i^{-1}x) \le 0, p_i(T_i^{-1}x ) \le 0\}$.
\begin{lem} \label{lem: equiv set opt and poly ineq}
Consider the sets $P_i:=\{x \in \R^n: p_i(x) \le 0,F_i(x) \le 0 \}$ and $C:=\{x \in \R^n: c(x) < 0, F_0(x)<0 \}$. Then Opt.~\eqref{opt: general packing} is equivalent to Opt.~\eqref{opt: poly ineq}.
\end{lem}
\begin{proof} 
	To show the optimization problems are equivalent we show that the constraints are equivalent. We start with the constraint that the object is contained within the container set.
	\begin{align*}
	&	T_i(P_i) \subset C \iff \text{If } x \in T_i(P_i)  \text{ then } x \in C  \iff x \in C \text{ for all } x \in T_i(P_i)\\
	& \iff F_0(x)<0 \text{ and } c(x)<0 \text{ for all } x \in \{ y \in \R^n :F_i(T_i^{-1} y) \le 0, p_i(T_i^{-1}y ) \le 0\}. 
	\end{align*}
	Therefore $	T_i(P_i) \subset C$ for $i \in \{1,\dots,N\}$ holds iff the first two constraints of Opt.~\eqref{opt: poly ineq} hold.
	
	We next show the equivalence of the non-overlapping object constraint by proof by contrapositive. Note that the contrapositive to $T_i(P_i) \cap T_j(P_j) = \emptyset$ is $T_i(P_i) \cap T_j(P_j)  \ne \emptyset$. Now,
	\begin{align}
	&	T_i(P_i) \cap T_j(P_j) \ne \emptyset \\ \nonumber
		& \iff  \exists  x \in \{y \in \R^n: F_i(T_i^{-1}y)\le0, p_i(T_i^{-1}y) \le 0,F_j(T_j^{-1}y) \le 0, p_j(T_j^{-1}y) \le 0 \}\\
		& \iff  \label{contra}
	 \exists x \in \{y \in \R^n: F_i(T_i^{-1}y) \le 0, p_i(T_i^{-1}y) \le 0,F_j(T_j^{-1}y) \le 0 \} \text{ and } p_j(T_j^{-1}x) \le 0.
	\end{align}
	The contrapositive of Eq.~\eqref{contra}  is exactly the third constraint of Opt.~\eqref{opt: poly ineq}. Since the contrapositive of $T_i(P_i) \cap T_j(P_j) = \emptyset$ is equal to the contrapositive of the third constraint of Opt.~\eqref{opt: poly ineq}, it follows that $T_i(P_i) \cap T_j(P_j) = \emptyset$ is equivalent to the third constraint of Opt.~\eqref{opt: poly ineq}.
	\end{proof}
\vspace{-0.5cm}
\subsection{A Sum-of-Squares Tightening for Certifying Correct Packing Configurations}	
%
For a given configuration, fixed $\{T_i\}_{i=1}^N \subset E(n)$, we next tighten  Opt.~\eqref{opt: poly ineq} to an SOS problem,
\vspace{-0.2cm} \begin{align} \label{opt: SOS config certification}
	&\gamma^* := \arg \max_{\gamma \in \R, s_i^{(1)},s_i^{(2)},s_i^{(3)},s_i^{(4)},s_i^{(5)},s_i^{(6)},s_i^{(7)},s_i^{(8)},s_i^{(9)},s_i^{(10)} \in \sum_{SOS}^d} \gamma \\ \nonumber 
	& \text{Subject to: }\\ \nonumber 
	& \quad -c(x)  +p_i(T_i^{-1}x )s_i^{(1)}(x) +F_i(T_i^{-1}x) s_i^{(2)}(x) = s_i^{(3)}(x) +\gamma  \text{ for all }  i \in \{1,...,N\},  \\ \nonumber 
		& \quad -F_0(x)  +p_i(T_i^{-1}x )s_i^{(4)}(x) +F_i(T_i^{-1}x) s_i^{(5)}(x) = s_i^{(6)}(x) +\gamma  \text{ for all }  i \in \{1,...,N\},  \\ \nonumber 
	& \quad p_j(T_j^{-1}x ) + p_i(T_i^{-1}x) s_i^{(7)}(x)+ F_i(T_i^{-1}x)s_i^{(8)}(x)+ F_j(T_j^{-1}x) s_i^{(9)}(x)= s_i^{(10)}(x) + \gamma \text{ for all }  i \ne j.
\end{align}

Opt.~\eqref{opt: SOS config certification} is a feasibility SOS program. Rather than certifying all the constraints simultaneously we can decompose the problem to certify each of the constraints one at a time (or in parallel): 
\begin{align} \label{opt: SOS c1}
	&\gamma^{(1)}_{i} := \arg \max_{\gamma \in \R, s_1,s_2,s_3 \in \sum_{SOS}^d} \gamma \\ \nonumber 
	& \text{Subject to: }-c(x)  +p_i(T_i^{-1}x )s_1(x) +F_i(T_i^{-1}x) s_2(x) = s_3(x) +\gamma  \\ \label{opt: SOS c2}
	&\gamma^{(2)}_{i} := \arg \max_{\gamma \in \R, s_1,s_2,s_3 \in \sum_{SOS}^d} \gamma \\ \nonumber 
	& \text{Subject to: } -F_0(x)  +p_i(T_i^{-1}x )s_1(x) +F_i(T_i^{-1}x) s_2(x) = s_{3}(x) +\gamma \\ \label{opt: non_overlap}
	&\gamma^{(3)}_{i,j} := \arg \max_{\gamma \in \R, s_1,s_2,s_3,s_4 \in \sum_{SOS}^d} \gamma \\ \nonumber 
	& \text{Subject to: } p_j(T_j^{-1}x ) + p_i(T_i^{-1}x) s_1(x)+ F_i(T_i^{-1}x)s_2(x)+ F_j(T_j^{-1}x) s_3(x)= s_4(x) + \gamma.
\end{align}

\begin{thm} \label{thm: cert correct packing}
Suppose $P_i:=\{x \in \R^n: p_i(x) \le 0,F_i(x) \le 0 \}$ and $C:=\{x \in \R^n: c(x) < 0, F_0(x)<0 \}$, where the sublevel set $\{x \in \R^n: F_i(x) \le 0 \}$ is compact for each $i \in \{1,\dots,N\}$. Then, $\{T_i\}_{i=1}^N \subset E(n)$ is a feasible solution to Opt.~\eqref{opt: general packing} if and only if there exists $d \in \N$ such that the solutions of Opts~\eqref{opt: SOS c1}, \eqref{opt: SOS c2} and~\eqref{opt: non_overlap} are positive, that is $\gamma^{(1)}_i>0$ for $i \in \{1,\dots,N\}$, $\gamma^{(2)}_i>0$ for $i \in \{1,\dots,N\}$ and $\gamma^{(3)}_{i,j}>0$ for $i \ne j$ and $i,j \in \{1,\dots,N\}$.
\end{thm}
\begin{proof}
The direction that if there exists $d \in \N$ such that the solutions of Opts~\eqref{opt: SOS c1}, \eqref{opt: SOS c2} and~\eqref{opt: non_overlap} are positive implies $\{T_i\}_{i=1}^N \subset E(n)$ is a feasible solution to Opt.~\eqref{opt: general packing} follows trivially since Opts~\eqref{opt: SOS c1}, \eqref{opt: SOS c2} and~\eqref{opt: non_overlap} are tightening of Opt.~\eqref{opt: poly ineq} and thus $\{T_i\}_{i=1}^N \subset E(n)$ is feasible for Opt.~\eqref{opt: poly ineq}. Lemma~\ref{lem: equiv set opt and poly ineq} then shows that $\{T_i\}_{i=1}^N \subset E(n)$ is feasible for Opt.~\eqref{opt: general packing}.
	
	Let us now prove the other direction. Suppose $\{T_i\}_{i=1}^N \subset E(n)$ is a feasible solution to Opt.~\eqref{opt: general packing}. Then by Lemma~\ref{lem: equiv set opt and poly ineq} $\{T_i\}_{i=1}^N \subset E(n)$ is a feasible solution to Opt.~\eqref{opt: poly ineq}. The constraints of Opt.~\eqref{opt: poly ineq} consist of polynomial inequalities, all being of the form: $f(x)>0$ for all $x \in X$, where $f \in \R[x]$ and {$X=\{x \in \R^n: g_1(x) \ge 0,\dots,g_m(x) \ge 0\}$} is a compact semialgebraic set (note, Euclidean transformations are distance preserving so if $P_i$ is compact then so is $T_i(P_i)$). For simplicity we next only show $\gamma^{(1)}_i>0$, since all the solutions to other instances of Opts~\eqref{opt: SOS c1}, \eqref{opt: SOS c2} and~\eqref{opt: non_overlap} can be shown to be positive by a similar argument. For this case,  $f(x)=-c(x)$, $g_1(x)=-p_i(T_i^{-1}x)$ and $g_2(x)=-F_i(T_i^{-1}x)$. Because $X$ is compact and $f$ is continuous, the extreme value theorem implies that there exists $x^* \in X$ where the minimum of $f$ over $X$ is attained, that is $\min_{x \in X}f(x)=f(x^*)$. Let $\eta:=0.1f(x^*)$. Since $f$ is strictly positive over $X$ and $x^* \in X$ it follows that $\eta=f(x^*)>0$. Now $f(x)-\eta=f(x)-0.1f(x^*) \ge  \min_{x\in X}\{f(x)\}-0.1f(x^*)=0.9 f(x^*)>0$ for all $x \in X$. By Putinar's Positivstellesatz (Thm.~\ref{thm: Psatz}) it follows there exists $\{s_i\}_{i=1}^m$ such that $	f(x) - \sum_{i=1}^m s_i(x) g_i(x) =s_{m+1}(x) + \eta$. Therefore for $d:= \max\{ \deg(s_i) \}$ it follows that $\eta>0$ is feasible to Opt.~\eqref{opt: SOS c1}. Since  $\gamma_i^{(1)}$ is defined as the optimal solution to Opt.~\eqref{opt: SOS c1} and we maximize over $\gamma$ it follows that $\gamma_i^{(1)} \ge \eta >0$.
	\end{proof}

Note, in this section we considered objects of the form $P_i:=\{x \in \R^n: p_i(x) \le 0,F_i(x) \le 0 \}$ because our approximate polynomial sublevel set representations found by solving Opt.~\eqref{opt: SOS discrete points} took this form (see Eq.~\eqref{eq: obj poly sublevel set rep}). This methodology can trivially be extended to the case where both objects and containers are described by general semialgebriac sets rather than the intersection of only two polynomial sublevel sets. Also note, in the simpler case of certifying correct packing when $P_i:=\{x \in \R^n: p_i(x) \le 0 \}$ and $C=\{x\in \R^n:c(x) \le 0 \}$ we simply set $F_i=0$ in Opts~\eqref{opt: SOS c1} and \eqref{opt: non_overlap} and no longer need to solve Opt.~\eqref{opt: SOS c2}.

\begin{rem}
	Theorem~\ref{thm: cert correct packing} shows that if the container and object sets are described by compact polynomial sublevel sets then there exists SOS programs capable of certifying the correctness of a given configuration.  This implies that correct packing configurations can be validated with polynomial time complexity, following from the fact that SOS programs can be reformulates as SDP programs that can be solved to arbitrary accuracy in polynomial time complexity~\cite{laurent2005semidefinite}. 
\end{rem}

Theorem~\ref{thm: cert correct packing} establishes that certifying the correctness of a packing configuration can be achieved by demonstrating positive solutions for Opts~\eqref{opt: SOS c1}, \eqref{opt: SOS c2}, and~\eqref{opt: non_overlap}. In numerical investigations, it was noticed that incorporating the constraint $\gamma \le 1$ into each optimization problem, preventing solutions from reaching excessively large values, enhances the performance of the SDP solver. While this artificial constraint does impact our ability to certify the correctness of a packing configuration, as certification merely necessitates positive solutions to Opts~\eqref{opt: SOS c1}, \eqref{opt: SOS c2}, and~\eqref{opt: non_overlap}, and positivity is independent of magnitude. Consequently, in the subsequent section, we will include the constraint $\gamma \le 1$ when certifying the correctness of packing configurations.
	
 \vspace{-0.2cm}
\section{Numerical Experiments}
\vspace{-0.2cm}
In this section we solve Opt.~\eqref{opt: SOS discrete points} for various point clouds and prior information to learn the shapes of several objects. We then arrange these objects into different configurations and solve Opts~\eqref{opt: SOS c1}, \eqref{opt: SOS c2} and~\eqref{opt: non_overlap} to certify the correctness of each packing configuration.  All SOS programs are
solved using Yalmip~\cite{lofberg2004yalmip} with SDP solver Mosek~\cite{aps2019mosek}.
\begin{ex}[Learning a crisp packet] \label{ex: learn crisp}
	Figure~\ref{fig:crisp} shows the process of converting an image of a packet of crisps into a polynomial sublevel set. Fig.~\ref{subfig:pic Crisp_plot} shows the original image of the crisps captured by a standard mobile phone camera. The point cloud associated with this image is then depicted in Fig.~\ref{subfig:pic pointcloud_crisp}, where the point cloud is generated by randomly sampling discrete points for positions of pixels that are not coloured white. Figs~\ref{subfig:pic crisp_poly},\ref{subfig:pic crisp_poly_sym},\ref{subfig:pic crisp_poly_star} and~\ref{subfig:pic crisp_poly_convex} show the 0-sublevel sets of the $d=18$, {$R=1.66$ and $\Lambda=[-1.1,1.1]^2$} solution to Opt.~\eqref{opt: SOS discrete points} with $\mathcal{S}=\emptyset$, $\mathcal{S}$ given in Eq.~\eqref{prior knowledge: symetry} and $A=-I$, $\mathcal{S}$ given in Eq.~\eqref{prior knowledge: star} and $\mathcal{S}$ given in Eq.~\eqref{prior knowledge: convexity} respectively.   The sublevel sets in Figs~\ref{subfig:pic crisp_poly} and~\ref{subfig:pic crisp_poly_sym} are connected but not simply connected since they possess holes. However, as expected, the sublevel set in Fig.~\ref{subfig:pic crisp_poly_sym} is symmetrical about the lines $x_2=x_1$ and $x_2=-x_1$. On the other hand, both Figs~\ref{subfig:pic crisp_poly_star} and~\ref{subfig:pic crisp_poly_convex} are simply connected, possessing no holes. However, Fig.~\ref{subfig:pic crisp_poly_star} is star connected and not convex whereas Fig.~\ref{subfig:pic crisp_poly_convex} is convex as expected. Interestingly, the original image is represented by a $ 1600   \times  1252   \times 3$  matrix, whose elements provide positions and colours of each pixel, whereas, the polynomial sublevel set representation is stored using only a length $190$ vector representing the coefficients of the degree $18$ polynomial. We have effectively compressed the information stored in the image into $190$ numbers that allows us to easily model what happens when we rotate and shift the object.
	
	\begin{figure}[htbp] 
		\floatconts
		{fig:crisp}
		{\vspace{-25pt} \caption{ Polynomial representations of a crisp packet associated with Ex.~\ref{ex: learn crisp}. Fig.~\ref{subfig:pic Crisp_plot} shows crisp image. Fig.~\ref{subfig:pic pointcloud_crisp} shows point cloud derived from pixel data. Figs~\ref{subfig:pic crisp_poly},\ref{subfig:pic crisp_poly_sym},\ref{subfig:pic crisp_poly_star} and~\ref{subfig:pic crisp_poly_convex} show the 0-sublevel sets of the $d=18$ solution to Opt.~\eqref{opt: SOS discrete points} with $\mathcal{S}=\emptyset$, $\mathcal{S}$ given in Eq.~\eqref{prior knowledge: symetry} and $A=-I$, $\mathcal{S}$ given in Eq.~\eqref{prior knowledge: star} and $\mathcal{S}$ given in Eq.~\eqref{prior knowledge: convexity} respectively.    \vspace{-30pt}}}  
		{%
			\subfigure{%
				\label{subfig:pic Crisp_plot}
				\includegraphics[width=0.25 \linewidth, trim = {3cm 0cm 3cm 0cm}, clip]{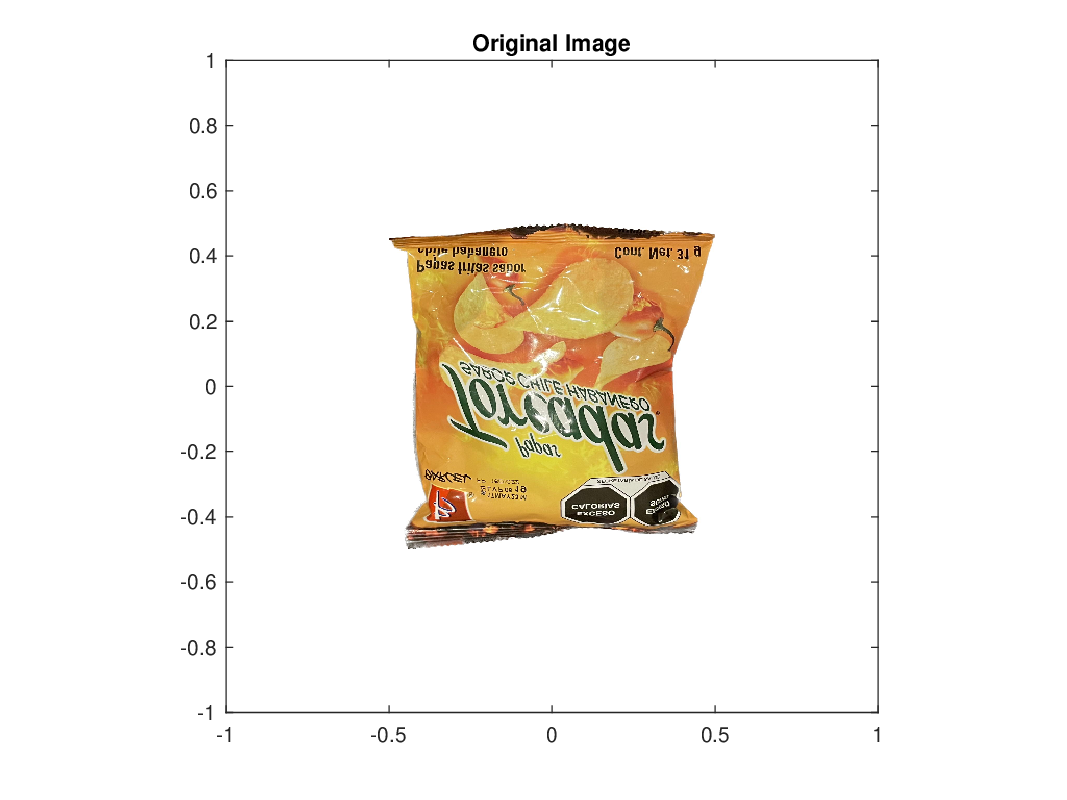}
			} 
			\subfigure{%
				\label{subfig:pic pointcloud_crisp}
				\includegraphics[width=0.25 \linewidth, trim = {3cm 0cm 3cm 0cm}, clip]{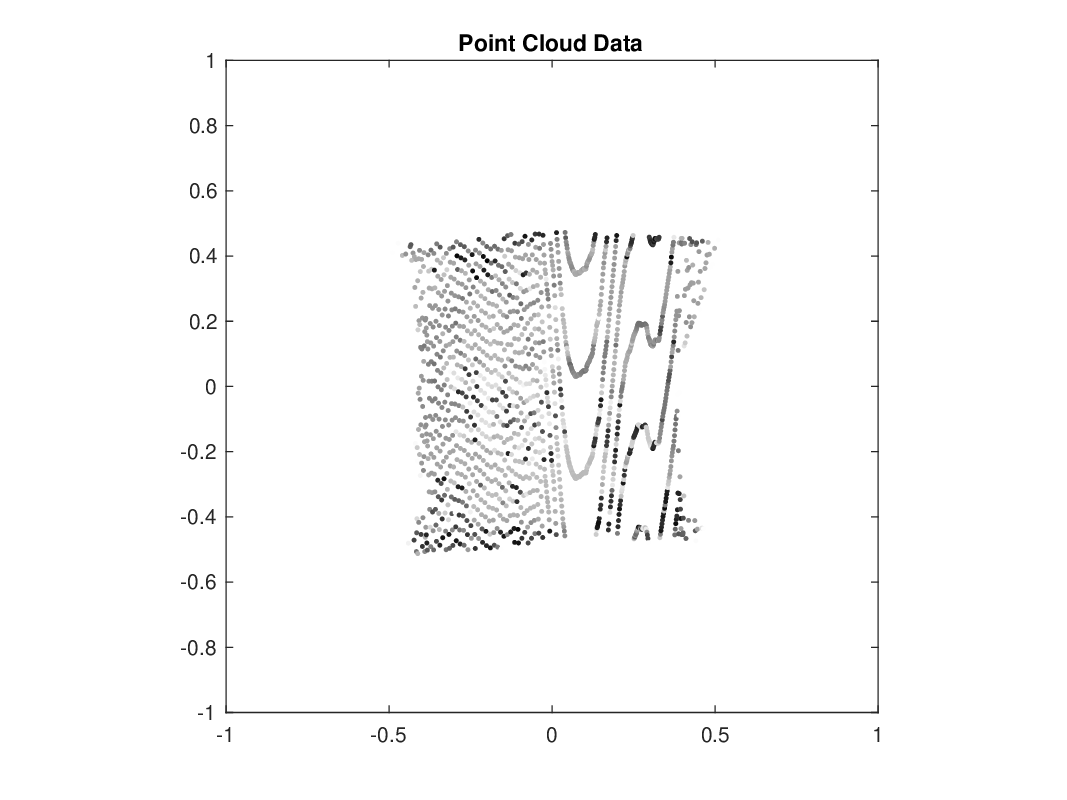}
			}
				\subfigure{%
			\label{subfig:pic crisp_poly}
			\includegraphics[width=0.25 \linewidth, trim = {3cm 0cm 3cm 0cm}, clip]{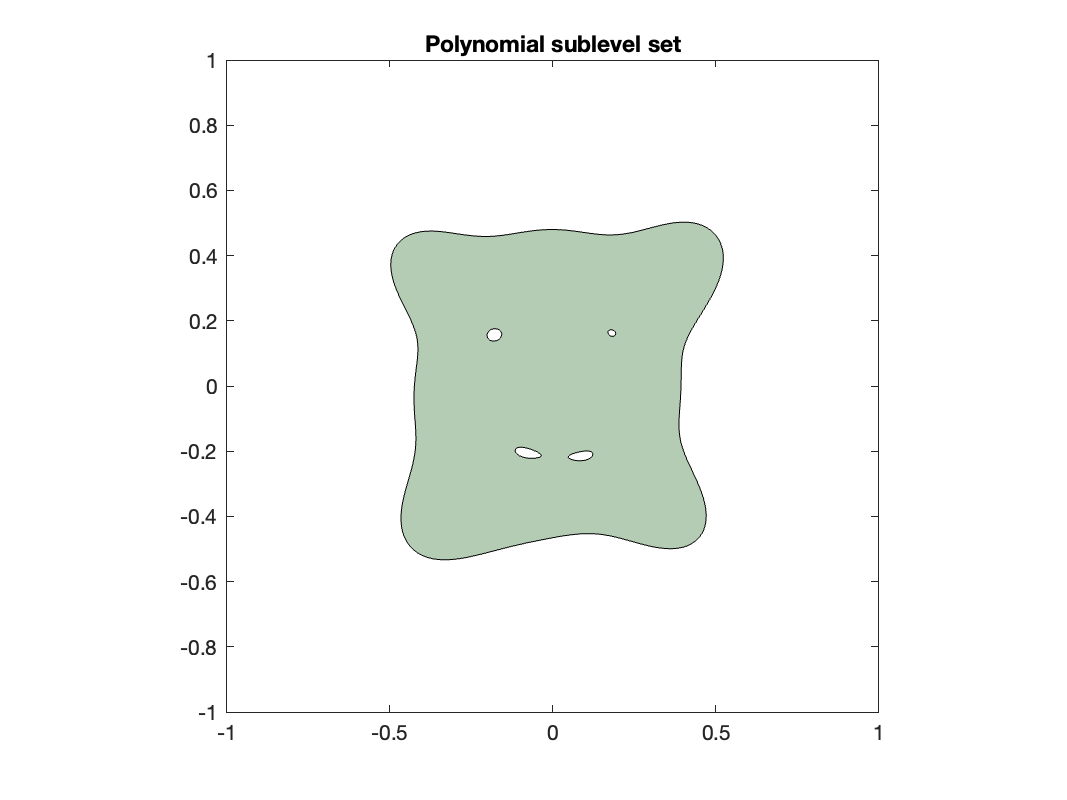}
		}
			\subfigure{%
		\label{subfig:pic crisp_poly_sym}
		\includegraphics[width=0.25 \linewidth, trim = {3cm 0cm 3cm 0cm}, clip]{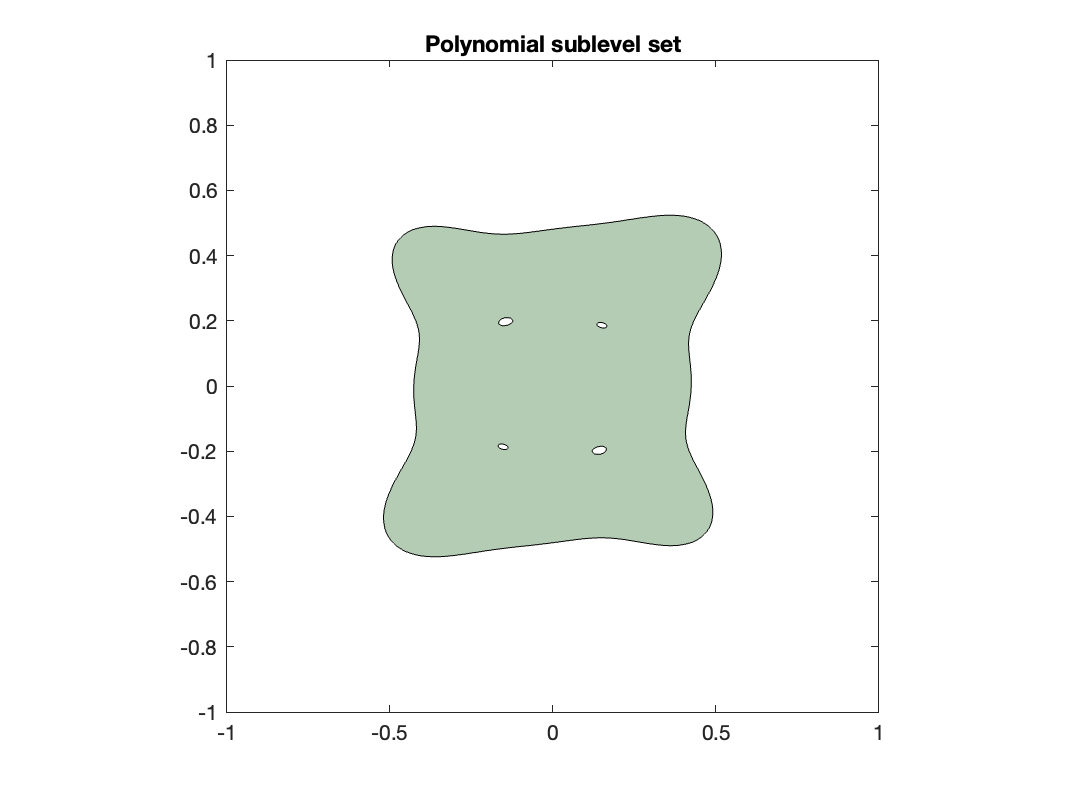}
	}
			\subfigure{%
	\label{subfig:pic crisp_poly_star}
	\includegraphics[width=0.25 \linewidth, trim = {3cm 0cm 3cm 0cm}, clip]{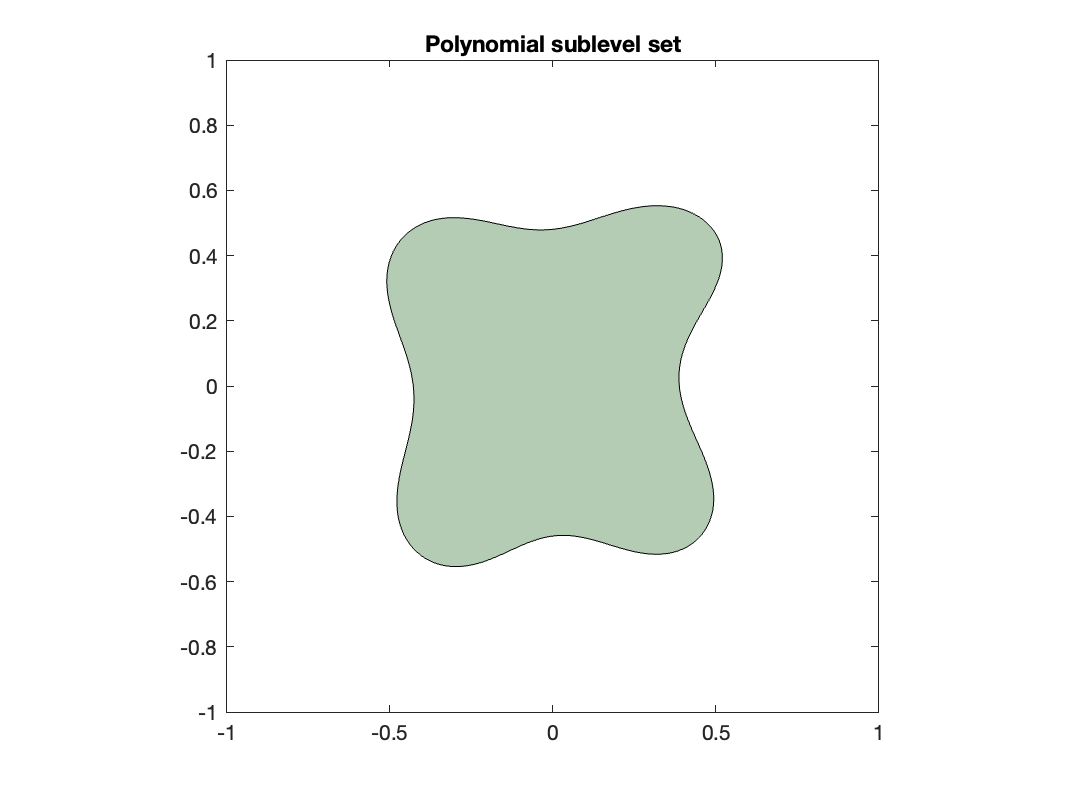}
}
			\subfigure{%
	\label{subfig:pic crisp_poly_convex}
	\includegraphics[width=0.25 \linewidth, trim = {3cm 0cm 3cm 0cm}, clip]{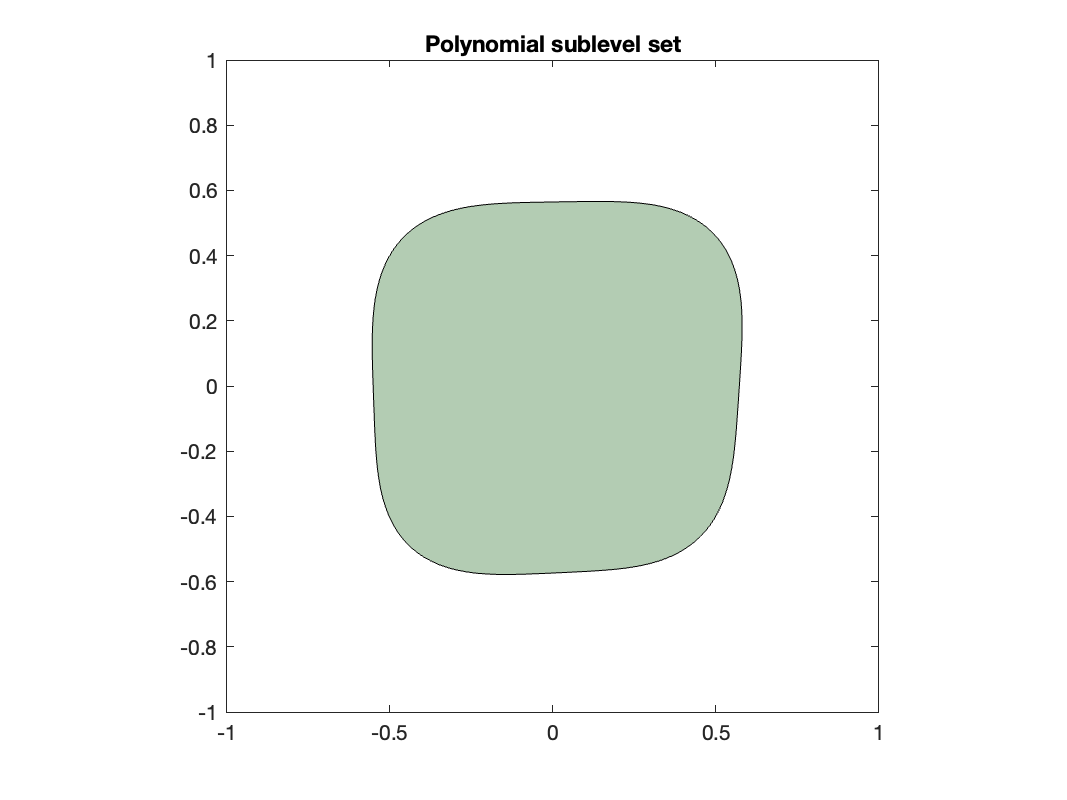}
}
		}
	\end{figure}

\end{ex}

\begin{ex}[Learning a teapot] \label{ex: teapot}
	Figure~\ref{fig:teapot} shows various polynomial representations of the Newell teapot. The point cloud data displayed in Fig.~\ref{subfig: teapot_pointcloud} was generated using Matlab's function \\ \texttt{teapotGeometry}. The approximate teapot sets in Figs~\ref{subfig: teapot_poly} and~\ref{subfig:teapot_poly_star} are the sublevel sets of solutions to Opt.~\ref{opt: SOS discrete points} for $d=20$, $\Lambda=[-1.2,1.2]^3$ and $R=2.18$ with the Fig.~\ref{subfig: teapot_poly} having $\mathcal{S}=\emptyset$ and Fig.~\ref{subfig:teapot_poly_star} having $\mathcal{S}$ given in Eq.~\eqref{prior knowledge: star}.
	
	\begin{figure}[htbp]
	\floatconts
{fig:teapot}
	{ \vspace{-25pt} \caption{  Polynomial representations of a teapot associated with Example~\ref{ex: teapot}. Fig.~\ref{subfig: teapot_pointcloud} shows shows point cloud data set of a teapot. Figs~\ref{subfig: teapot_poly} and~\ref{subfig:teapot_poly_star} shows 0-sublevel set of the $d=20$ solution to Opt.~\eqref{opt: SOS discrete points} with $\mathcal{S}=\emptyset$ and $\mathcal{S}$ given in Eq.~\eqref{prior knowledge: star} respectively. \vspace{-25pt} }}
	{%
		\subfigure{%
			\label{subfig: teapot_pointcloud}
			\includegraphics[width=0.3 \linewidth, trim = {1cm 0cm 1cm 0cm}, clip]{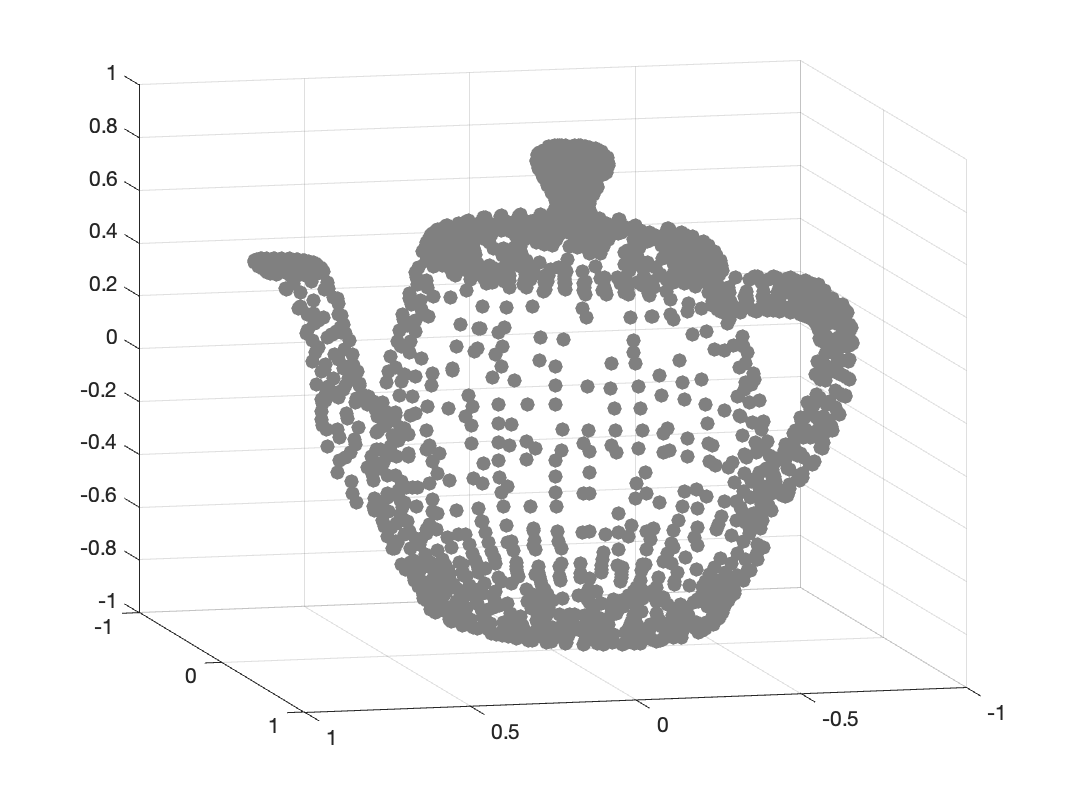}
		} 
		\subfigure{%
			\label{subfig: teapot_poly}
			\includegraphics[width=0.3 \linewidth, trim = {1cm 0cm 1cm 0cm}, clip]{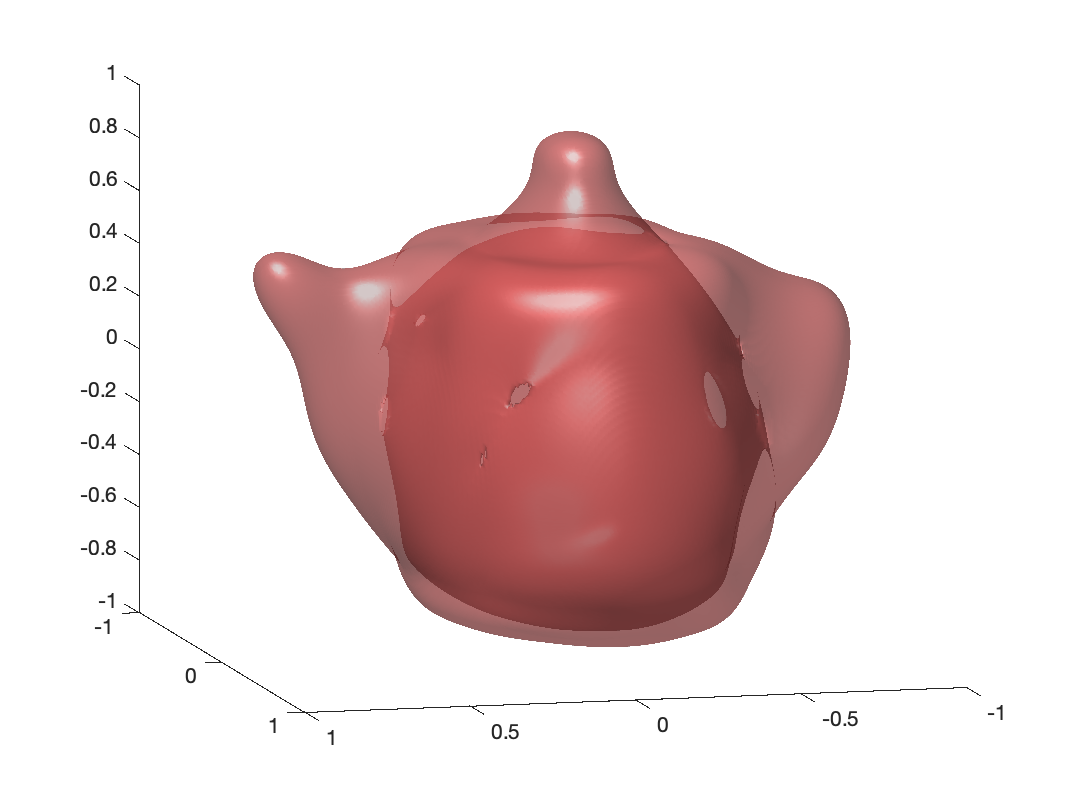}
		}
		\subfigure{%
			\label{subfig:teapot_poly_star}
			\includegraphics[width=0.3 \linewidth, trim = {1cm 0cm 1cm 0cm}, clip]{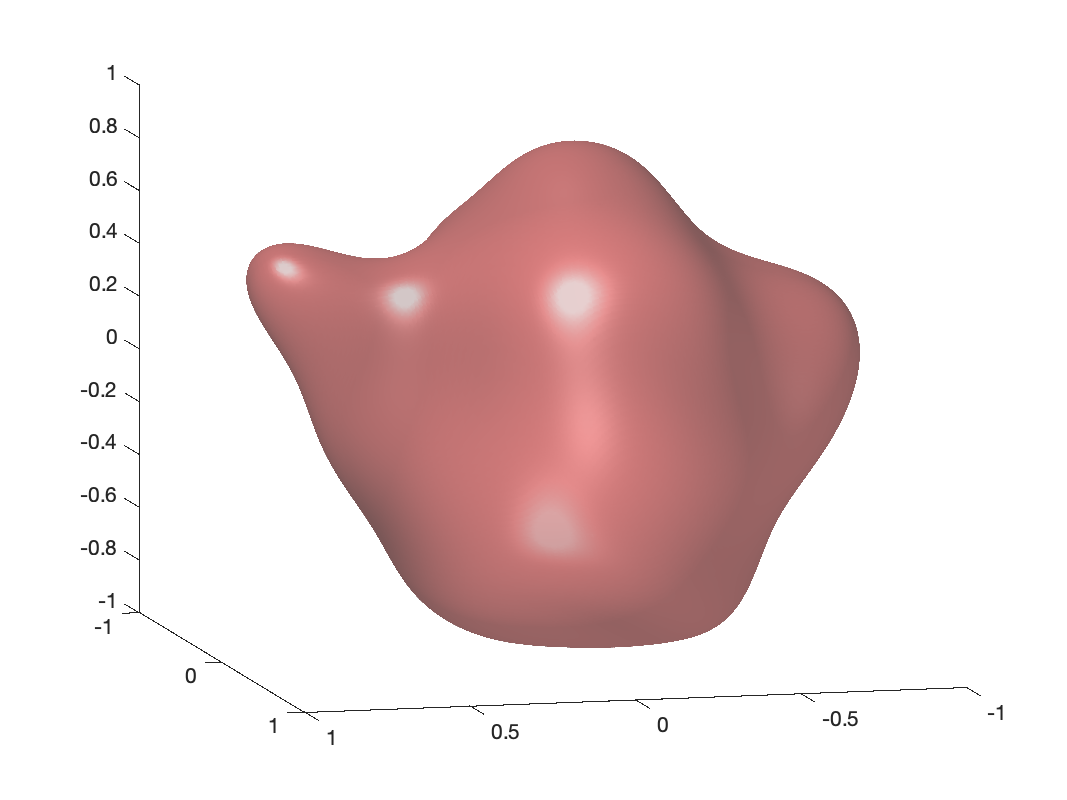}
		}
	}
\end{figure}
\end{ex}

\begin{ex}[Certifying the correctness of several packing configurations of teapots in a torus] \label{ex: cert teapots}
	Let us consider the problem of certifying the correct packing configurations of a torus type container set given as the $0$-sublevel set of $c(x)=(||x||_2^2 + 1)^3 - 10(x_1^2 + x_2^2)(x_3^2 + 1)$. Four objects are found by first solving Opt.~\eqref{opt: SOS discrete points} for $R=2.18$, $\Lambda=[-1.2,1.2]^3$, $d=6$, $\mathcal{S}=\emptyset$ and discrete points, $\{x_i\}$,  given by Matlab's function \texttt{teapotGeometry} (the same as Example~\ref{ex: teapot}) to find a function $J_6^*$. Objects are then given by $P_i=\{x \in R^n:J_6^*(T_i^{-1}x))<0,F(T_i^{-1}x)<0 \}$ where $F(x)=||x||_2^2-1$ and $T_i^{-1}x=3(x-c_i)$. Note, in this example, it is unnecessary to solve Opt.~\eqref{opt: SOS c2} to certify packing correctness since $C$ is described by a single sublevel set. In Fig.~\ref{subfig: config 1} centres of objects are given as follows, $c_1=[1, 1, 0]^\top$, $c_2=[-0,.4, 0, 0]^\top$, $c_3=[-0.5, 0.5, 0]^\top$ and $c_4=[-0.5, -0.5, 0]^\top$. Clearly, in this instance $P_1$ is outside of the container set and $P_2$ overlaps with both $P_3$ and $P_4$. Attempting to solve Opt.~\eqref{opt: SOS c1} for this incorrect object configuration with $i=1$ and $d=8$ results in the SDP solver, Mosek, outputting numerical problems and $\gamma_1^{(1)}= -23.3$. A similar situation occurs when attempting to solve Opt.~\eqref{opt: non_overlap} for $i=2$, $j=3$ and $d=8$, Mosek returns a value of $\gamma_{2,3}^{(3)}=-0.26$. On the other hand if we translate Objects 1 and 2 by making the coordinate changes $x\to x-[-0.5,-0.5,0]^\top$ and $x\to x-[0.9,-0.5,0]^\top$ to each object respectively we get the correct packing configuration displayed in Fig.~\ref{subfig: config 2}. This configuration can be certified as correct by using Theorem~\ref{thm: cert correct packing} and solving Opts~\eqref{opt: SOS c1} and~\eqref{opt: non_overlap} for $d=8$, where the smallest $\gamma$ was found to be $0.15>0$.
\begin{figure}[htbp]
	\floatconts
	{fig:teapot config}
	{ \vspace{-25pt} \caption{  Plot associated with Examples~\ref{ex: cert teapots} and~\ref{ex: cert crisps}. Fig.~\ref{subfig: config 1} shows an incorrect 3D packing configurations in which objects can be translated to a correct packing configuration given in Fig.~\ref{subfig: config 2}.  Fig.~\ref{subfig: crisp incorrect} shows an incorrect 2D packing configurations in which objects can be rotated to a correct packing configuration given in Fig.~\ref{subfig: crisp correct}. \vspace{-15pt} }}
	{%
		\subfigure{%
			\label{subfig: config 1}
			\includegraphics[width=0.239 \linewidth, trim = {1cm 0cm 1cm 0cm}, clip]{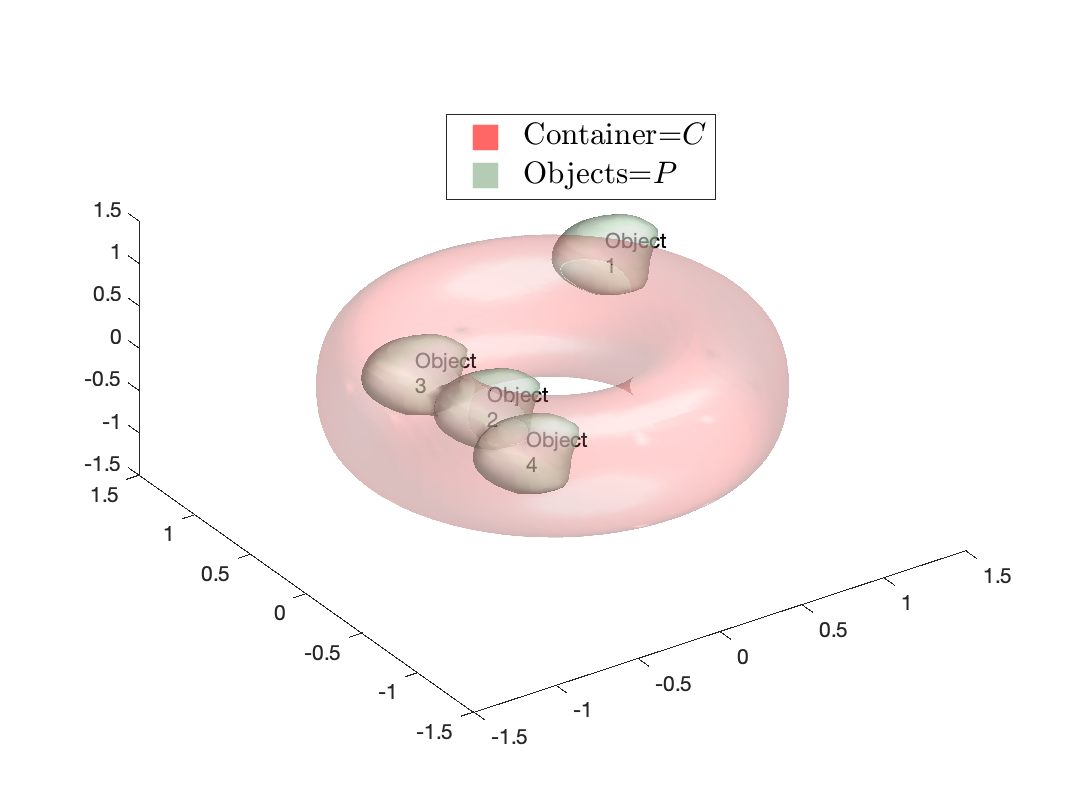}
		} 
			\subfigure{%
		\label{subfig: config 2}
		\includegraphics[width=0.239 \linewidth, trim = {1cm 0cm 1cm 0cm}, clip]{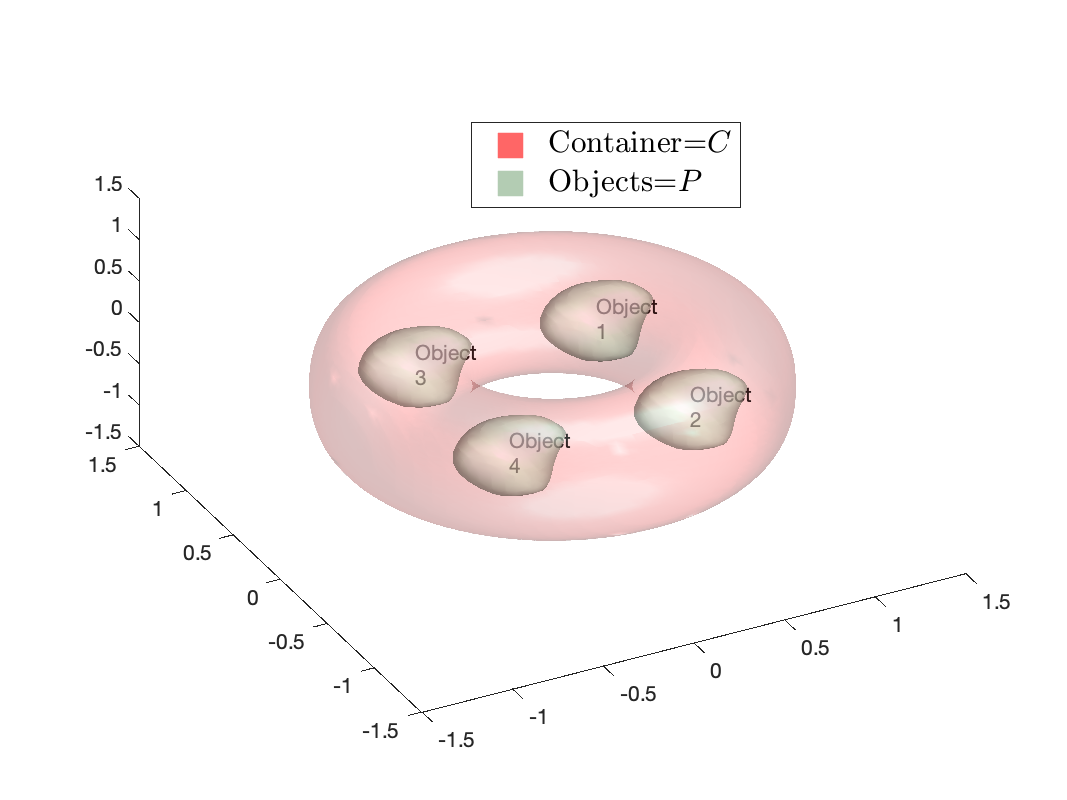}
	} 
			\subfigure{%
	\label{subfig: crisp incorrect}
	\includegraphics[width=0.239 \linewidth, trim = {1cm 0cm 1cm 0cm}, clip]{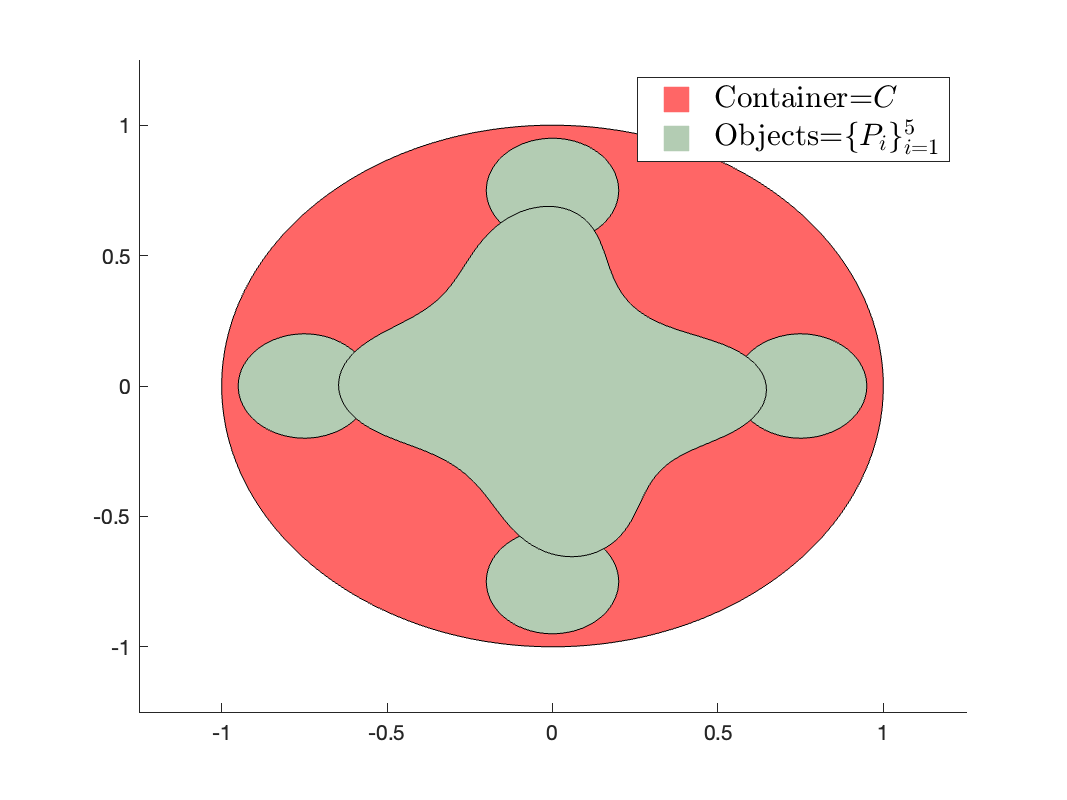}
} 
			\subfigure{%
	\label{subfig: crisp correct}
	\includegraphics[width=0.239 \linewidth, trim = {1cm 0cm 1cm 0cm}, clip]{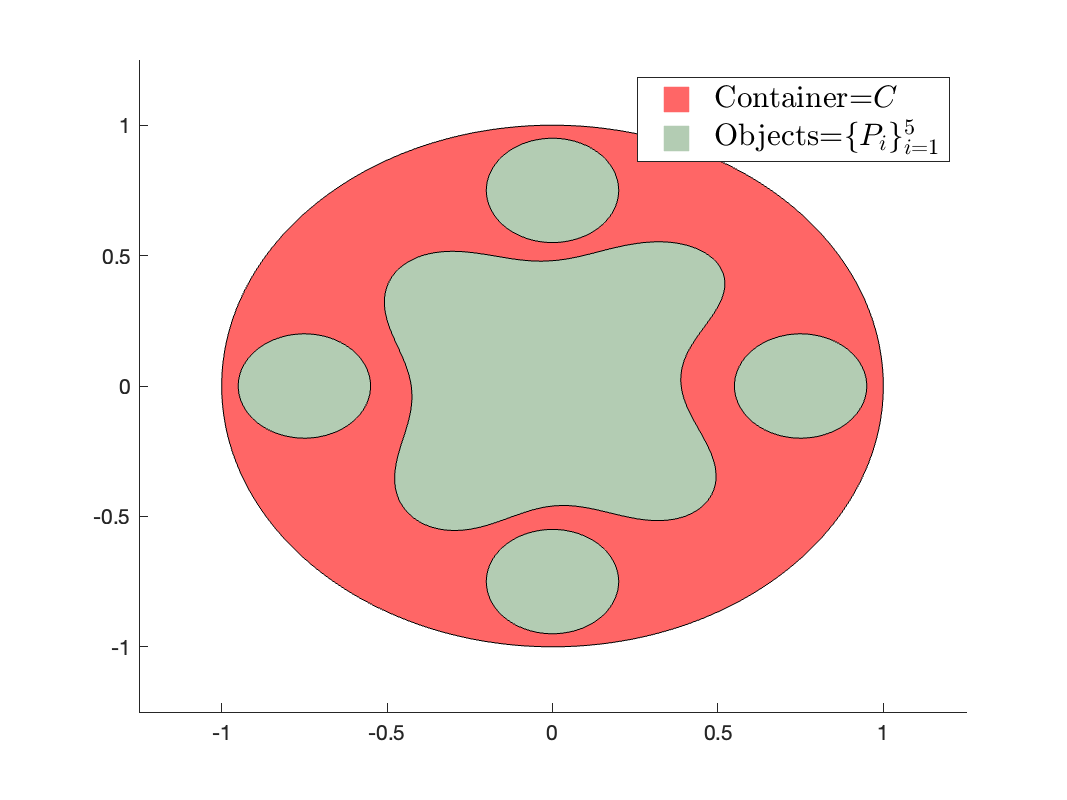}
} 	

}
\end{figure}

%

\end{ex}
\vspace{-0.2cm}
\begin{ex}[Certifying the packing correctness configurations of a crisp packet and circles] \label{ex: cert crisps}
	
	\noindent
	Let us consider the problem of certifying the correct packing configurations of a circular type container set, given by the $0$-sublevel set of $c(x)=||x||_2^2-1$. Five objects are given by $\{x\in \R^2: p_i(x)<0  \} $, where $p_1$ was found by solving Opt.~\eqref{opt: SOS discrete points} for the same crisp point cloud data from Ex.~\ref{ex: learn crisp} with {$d=18$, $r=1.66$, $\Lambda=[-1.1,1.1]^2$ and $\mathcal{S}$ given in Eq.~\eqref{prior knowledge: star}}. The other four objects are given by $p_i(x)=(x_1 \pm 0.75)^2 + x_2^2-0.2^2$ or $p_i(x)= x_1^2+(x_2 \pm 0.75)^2-0.2^2$. Note, in this example it is unnecessary to define a computation domain using $F_i$ as all sublevel sets are simply connected. Therefore to certify packing we set $F_i \equiv 0$ in Opt.~\eqref{opt: SOS c1} and Opt.~\eqref{opt: non_overlap} and no longer need to solve Opt.~\eqref{opt: SOS c2}. The initial packing configuration given in Fig.~\ref{subfig: crisp incorrect} is incorrect. Attempting to solve solve Opt.~\eqref{opt: non_overlap} for this incorrect object configuration with{$i=1$, $j=2$ and $d=22$} results in the SDP solver, Mosek, outputting numerical problems and {$\gamma_{1,2}^{(3)}=-0.03$}. On the other hand, if we rotate $P_1$ using the rotation matrix $R = \begin{bmatrix}
		\cos(\frac{\pi}{4}) & -\sin(\frac{\pi}{4}) \\
		\sin(\frac{\pi}{4}) & \cos(\frac{\pi}{4})
	\end{bmatrix} \in E(2)$ we get the correct packing configuration depicted in Fig.~\ref{subfig: crisp correct}. This configuration can be certified as correct by using Theorem~\ref{thm: cert correct packing} and solving Opts~\eqref{opt: SOS c1} and~\eqref{opt: non_overlap} for {$d=22$}, where the smallest $\gamma$ was found to be {$0.0071>0$.}


%
	\end{ex}
\vspace{-0.2cm}
  \section{Conclusion}
  \vspace{-0.2cm}
Leveraging object point cloud data, we've proposed and implemented a method grounded in convex optimization to learn polynomial sublevel set representations for target objects with known prior shape knowledge constraints. Through these polynomial representations, we have shown that for a sufficiently large degree there exists an SOS program that can certify correct packing configurations. In future work, we aim to integrate our SOS packing certification program into a reinforcement learning framework, autonomously incentivizing agents suggesting correct packing configurations.


\vspace{-0.2cm}
	\section{Appendix}
	\begin{thm}[Putinar's Positivstellesatz \citep{putinar1993positive}] \label{thm: Psatz}
		Consider the semialgebriac set $X = \{x \in \R^n: g_i(x) \ge 0 \text{ for } i=1,...,k\}$. Further suppose $\{x  \in \R^n : g_i(x) \ge 0 \}$ is compact for some $i \in \{1,..,k\}$. If the polynomial $f: \R^n \to \R$ satisfies $f(x)>0$ for all $x \in X$, then there exists SOS polynomials $\{s_i\}_{i \in \{1,..,m\}} \subset \sum_{SOS}$ such that $	f - \sum_{i=1}^m s_ig_i \in \sum_{SOS}.$ 	\end{thm}
%

	\begin{lem}[Compact sets can be approximated arbitrarily well using polynomial sublevel sets] \label{lem: any compact set can be approx by poly}
Suppose $X \subset B_r(0)$ for some $r>0$. Then for any $\eps>0$ there exists $J \in \R[x]$ such that $D_V(X,\{x\in B_r(0):J(x) \le 0 \} )<\eps$, where $D_V(A,B):=\mu\bigg( (A/B) \cup (B/A) \bigg)$ is the volume metric and $\mu$ is the Lebesgue measure. 
	\end{lem}
\begin{proof}
By Thm.~2.29 in~\cite{lee2012smooth}, for any compact set $X \subset \R^n$ there exists a smooth function $V$ such that $X = \{x \in \R^n: V(x)  \le 0\}=\{x \in B_r(0): V(x) \le 0\}$. By the Weierstrass approximation theorem, for any $\delta>0$ there exists $J_\delta \in \R[x]$ such that $|V(x)-J_\delta(x)|< \delta$ for all $x \in B_r(0)$. Let $\tilde{J}_\delta(x)=J_\delta(x)-\delta$ then $\tilde{J}_\delta(x)<V(x)$ and $|V(x)-\tilde{J}_\delta(x)|< 2\delta$. We have shown that there exists a sequence of polynomials that approximates $V$ from bellow. Hence by Proposition~1 from~\cite{jones2023sublevel} it follows that $\lim_{\delta \to 0} D_V(X,\{x\in B_r(0):J_\delta (x) \le 0 \} )$.
	\end{proof}
	
	
	\bibliography{bib.bib}
\end{document}